\documentclass[11pt,reqno]{amsart}
\usepackage{cmap}
\usepackage[T1]{fontenc}
\usepackage[utf8]{inputenc}
\usepackage[english]{babel}
\usepackage{amsfonts,amssymb,amsthm,amsmath}
\usepackage{url}
\usepackage{enumitem}

\usepackage{secdot}

\allowdisplaybreaks

\usepackage{graphicx}
\usepackage{epstopdf}
\usepackage{subfig}

\usepackage{a4wide}
\usepackage{xcolor}
\usepackage[colorlinks=true,linktocpage,pdfpagelabels,
bookmarksnumbered,bookmarksopen]{hyperref}
\definecolor{ForestGreen}{rgb}{0.1,0.6,0.05}
\definecolor{EgyptBlue}{rgb}{0.063,0.1,0.6}
\hypersetup{
	colorlinks=true,
	linkcolor=EgyptBlue,         
	citecolor=ForestGreen,
	urlcolor=olive
}

\usepackage[hyperpageref]{backref}

\newtheorem{theorem}{Theorem}

\newtheorem{lemma}[theorem]{Lemma}
\newtheorem{corollary}[theorem]{Corollary}

\theoremstyle{definition}
\newtheorem{definition}[theorem]{Definition}
\newtheorem{remark}[theorem]{Remark}

\let\OLDthebibliography\thebibliography
\renewcommand\thebibliography[1]{
	\OLDthebibliography{#1}
	\setlength{\parskip}{1pt}
	\setlength{\itemsep}{1pt plus 0.3ex}
}

\numberwithin{equation}{section}
\numberwithin{theorem}{section}
\numberwithin{equation}{section}
\numberwithin{theorem}{section}

\newcommand{\N}{\mathbb{N}}
\newcommand{\R}{\mathbb{R}}

\title[Basisness of Fu\v{c}\'ik eigenfunctions for the Dirichlet Laplacian]{Basisness of Fu\v{c}\'ik eigenfunctions for the Dirichlet Laplacian}

\author[F.~Baustian]{Falko Baustian}
\author[V.~Bobkov]{Vladimir Bobkov}

\address[F.~Baustian]{\newline\indent
	Institute of Mathematics, University of Rostock,
	\newline\indent 
	Ulmenstra{\ss}e 69, 18057 Rostock, Germany
}
\email{falko.baustian@uni-rostock.de}

\address[V.~Bobkov]{\newline\indent
	Institute of Mathematics, Ufa Federal Research Centre, RAS,
	\newline\indent 
	Chernyshevsky str.~112, 450008 Ufa, Russia
}
\email{bobkov@matem.anrb.ru}

\date{}


\thanks{
	The main part of the research was done during a stay of F.~Baustian at the Ufa Federal Research Centre. The stay was financed by the German-Russian Interdisciplinary Science Center (G-RISC), grant no.~F-2021b-8\_d.
	V.~Bobkov was supported in the framework of implementation of the development program of Volga Region Mathematical Center (agreement no.~075-02-2021-1393).
}

\begin{document}
	\begin{abstract} 
		We provide improved sufficient assumptions on sequences of Fu\v{c}\'ik eigenvalues of the one-dimensional Dirichlet Laplacian which guarantee that the corresponding Fu\v{c}\'ik eigenfunctions form a Riesz basis in $L^2(0,\pi)$.
		For that purpose, we introduce a criterion for a sequence in a Hilbert space to be a Riesz basis.
	\end{abstract}
	\maketitle 
	
	\section{Introduction}
	
	We study basis properties of sequences of eigenfunctions of the \textit{Fu\v{c}\'ik eigenvalue problem} for the one-dimensional Dirichlet Laplacian
	\begin{equation}\label{eq:fucik}
		\left\{
		\begin{alignedat}{1}
			-u''(x) &= \alpha u^{+}(x)-\beta u^{-}(x), ~~ x \in (0,\pi), \\
			u(0) &=u(\pi)=0,
		\end{alignedat}
		\right.
	\end{equation}
	where $u^{+}=\max(u,0)$ and $u^{-}=\max(-u,0)$.
	The Fu\v{c}\'ik spectrum is the set $\Sigma(0,\pi)$ of pairs $(\alpha,\beta)\in \R^2$ for which \eqref{eq:fucik} possesses a non-zero classical solution.
	Any $(\alpha,\beta) \in \Sigma(0,\pi)$ is called \textit{Fu\v{c}\'ik eigenvalue} and any corresponding non-zero classical solution of \eqref{eq:fucik} is called \textit{Fu\v{c}\'ik eigenfunction}.
	The Fu\v{c}\'ik eigenvalue problem \eqref{eq:fucik} was introduced in \cite{dancer} and \cite{fucik} to study elliptic equations
	with ``jumping'' nonlinearities, and afterwards it has been widely investigated in various aspects and for different operators, see, e.g., the surveys \cite{cuesta}, \cite[Chapter 9.4]{motreanu}, and references therein.
	To the best of our knowledge, basisness of sequences of Fu\v{c}\'ik eigenfunctions was considered for the first time in \cite{bb1}. 
	In that article, we provided several sufficient assumptions on sequences of Fu\v{c}\'ik eigenvalues 
	to obtain Riesz bases of $L^2(0,\pi)$ consisting of Fu\v{c}\'ik eigenfunctions.
	Let us recall that a sequence is a Riesz basis in a Hilbert space if it is the image of an orthonormal basis of that space under a linear homeomorphism, see, e.g., \cite{young}.
	The aim of the present note is to use more general techniques to significantly improve the results of \cite{bb1}.
	
	Let us describe the structure of the Fu\v{c}\'ik spectrum $\Sigma(0,\pi)$.
	It is not hard to see that the lines $\{1\}\times\R$ and $\R\times\{1\}$ are subsets of $\Sigma(0,\pi)$, since they correspond to sign-constant solutions of \eqref{eq:fucik} which are constant multiples of $\sin x$, the first eigenfunction of the Dirichlet Laplacian in $(0,\pi)$.
	The remaining part of $\Sigma(0,\pi)$ is exhausted by the hyperbola-type curves
	$$
	\Gamma_n
	=
	\left\{(\alpha,\beta)\in\R^2: \,
	\frac{n}{2}\frac{\pi}{\sqrt{\alpha}}
	+
	\frac{n}{2}\frac{\pi}{\sqrt{\beta}}
	=\pi
	\right\}
	$$
	for even $n \in \mathbb{N}$, and 
	\begin{align*}
		\Gamma_n &= \left\{(\alpha,\beta)\in\R^2: \,
		\frac{n+1}{2}\frac{\pi}{\sqrt{\alpha}}+\frac{n-1}{2}\frac{\pi}{\sqrt{\beta}}=\pi
		\right\},
		\\
		\widetilde{\Gamma}_n &=
		\left\{(\alpha,\beta)\in\R^2: \, \frac{n-1}{2}\frac{\pi}{\sqrt{\alpha}}+\frac{n+1}{2}\frac{\pi}{\sqrt{\beta}}=\pi
		\right\}
	\end{align*}
	for  odd $n \geq 3$, see, e.g.,  \cite[Lemma 2.8]{fucik}.
	Evidently, $(\alpha,\beta) \in \Gamma_n$ for odd $n \geq 3$ implies $(\beta,\alpha) \in \widetilde{\Gamma}_n$. 
	If $u$ is a Fu\v{c}\'ik eigenfunction for some $(\alpha,\beta)$, then so is $tu$ for any $t>0$, while $-tu$ is a Fu\v{c}\'ik eigenfunction for $(\beta,\alpha)$. 
	Hence, we neglect the curve $\widetilde{\Gamma}_n$ from our investigation of the basis properties of Fu\v{c}\'ik eigenfunctions.
	Each sign-changing Fu\v{c}\'ik eigenfunction consists of alternating positive and negative bumps, where positive bumps are described by $C_1\sin(\sqrt{\alpha}(x-x_1))$, while negative bumps are described by $C_2\sin(\sqrt{\beta}(x-x_2))$, for proper constants $C_1,C_2,x_1, x_2 \in \mathbb{R}$.
		
	We want to uniquely specify a Fu\v{c}\'ik eigenfunction for each point of $\Sigma(0,\pi)$. 
	In slight contrast to \cite{bb1}, we normalize Fu\v{c}\'ik eigenfunctions in such a way that they are ``close'' to the functions 
	$$
	\varphi_k(x) = \sqrt{\frac{2}{\pi}}\, \sin(kx),
	\quad 
	k \in \mathbb{N},
	$$ 
	which form a complete \textit{orthonormal} system in $L^2(0,\pi)$.
	This choice will be helpful in the proof of our main result, Theorem \ref{thm1}, below.
	\begin{definition}
		Let $n\geq2$ and $(\alpha,\beta)\in \Gamma_n$. The \textit{normalized Fu\v{c}\'ik eigenfunction} $g^n_{\alpha,\beta}$
		is the $C^2$-solution of the boundary value problem \eqref{eq:fucik} with $(g^n_{\alpha,\beta})'(0)>0$ and which is normalized by
		$$
		\|g^n_{\alpha,\beta}\|_{\infty}
		=
		\sup_{x\in[0,\pi]}|g^n_{\alpha,\beta}(x)|
		=
		\sqrt{\frac{2}{\pi}}.
		$$
		For $n=1$, we set $g^1_{\alpha,\beta}=\varphi_1$ for every $(\alpha,\beta)\in(\{1\}\times\R)\cup(\R\times\{1\})$.
	\end{definition}
	
	Piecewise definitions of the Fu\v{c}\'ik eigenfunctions $f^n_{\alpha,\beta}=\sqrt{\pi/2}\,g^n_{\alpha,\beta}$ can be found in the equations (1.2) and (1.3) in \cite{bb1}.
	In accordance to \cite{bb1}, we study the basisness of sequences of Fu\v{c}\'ik eigenfunctions described by the following definition.
	\begin{definition}\label{def:FS}
		We define the \textit{Fu\v{c}\'ik system} $G_{\alpha,\beta}=\{g^n_{\alpha(n),\beta(n)}\}$ as a sequence of normalized Fu\v{c}\'ik eigenfunctions with mappings $\alpha,\beta\colon\N\to\R$ satisfying $\alpha(1)=\beta(1)=1$ and $(\alpha(n),\beta(n))\in \Gamma_n$ for every $n\geq2$.
	\end{definition}
	
	We can now formulate our main result on the basisness of Fu\v{c}\'ik systems  which presents a non-trivial  generalization of \cite[Theorems 1.4 and 1.9]{bb1}.
	\begin{theorem}\label{thm1}
		Let $G_{\alpha,\beta}$ be a Fu\v{c}\'ik system. 
		Let $N$ be a subset of the even natural numbers and $N_* = \mathbb{N} \setminus N$.
		Assume that
		\begin{equation}\label{eq:sum<1}
			\sum_{n\in N_\ast}
			\left[
			1 - \frac{\langle g^n_{\alpha,\beta}, \varphi_n\rangle^2}{\|g^n_{\alpha,\beta}\|^2}
			\right]
			+
			E^2\left(
			\sup_{n \in N} \left\{\frac{4\max(\alpha(n),\beta(n))}{n^2}\right\}
			\right)
			< 1,
		\end{equation}
		with
		$\sup_{n \in N} \left\{4\max(\alpha(n),\beta(n))/n^2\right\} \in [4,9)$.
		Here, $E:[4,9) \to \mathbb{R}$ is a strictly increasing function defined as
		\begin{align}
			\notag
			E(\gamma)
			&=
			\frac{2\sqrt{2}}{\pi}\frac{\gamma^2}{\sqrt{\gamma}-1}\frac{(\sqrt{\gamma}-2)\sin\left(\frac{\pi}{\sqrt{\gamma}}\right) }{(\gamma-1)(2\sqrt{\gamma}-1)}
			+
			\frac{((3+\pi^2)\gamma+(9-2\pi^2)\sqrt{\gamma}-6)(\sqrt{\gamma}-2)}
			{3(\sqrt{\gamma}-1)(\sqrt{\gamma}+2)(3\sqrt{\gamma}-2)}
			\\
			\notag
			&+
			\frac{4}{\sqrt{3}\pi}\frac{\gamma^2}{\sqrt{\gamma}-1}\frac{(\sqrt{\gamma}-2)\sin\left(-\frac{3\pi}{\sqrt{\gamma}}\right)}{(9-\gamma)(2\sqrt{\gamma}-3)(4\sqrt{\gamma}-3)}
			+
			\frac2{\pi}\frac{\gamma^2}{\sqrt{\gamma}-1}\frac{(\sqrt{\gamma}-2)}{(16-\gamma)(3\sqrt{\gamma}-4)(5\sqrt{\gamma}-4)}
			\\
			\label{eq:E}
			&+
			\sqrt{\frac{6}{5}}\frac{2}{\pi}\frac{\gamma^2(\sqrt{\gamma}-2)}{\sqrt{\gamma}-1}
			\sum_{k=5}^\infty
			\frac{1}{(k^2-\gamma)((k-1)\sqrt{\gamma}-k)((k+1)\sqrt{\gamma}-k)}.
		\end{align}
		Then $G_{\alpha,\beta}$ is a Riesz basis in $L^2(0,\pi)$.
	\end{theorem}
	
	The proof of this theorem is given in Section \ref{sec:proof} and it is based on a general basisness criterion provided in Section \ref{sec:separ}.
	We visualize special cases of domains on the $(\alpha,\beta)$-plane described in Theorem \ref{thm1} in Figures \ref{fig:1} and \ref{fig1} below.
	
	Notice that, thanks to the orthonormality of $\{\varphi_n\}$, the terms in the first sum in \eqref{eq:sum<1} satisfy
	\begin{equation}\label{eq:estim}
		0 \leq 1 - \frac{\langle g^n_{\alpha,\beta}, \varphi_n\rangle^2}{\|g^n_{\alpha,\beta}\|^2}
		=
		\|g^n_{\alpha,\beta}-\varphi_n\|^2
		-
		\frac{(\|g^n_{\alpha,\beta}\|^2-\langle g^n_{\alpha,\beta},\varphi_n \rangle)^2}{\|g^n_{\alpha,\beta}\|^2}
		\leq 
		\|g^n_{\alpha,\beta}-\varphi_n\|^2,
	\end{equation}
	and we have the following explicit bounds:
	\begin{equation}\label{eq:Bn}
		\|g^n_{\alpha,\beta}-\varphi_n\|^2
		\leq 
		\left\{
		\begin{aligned}
			&\frac{8(3+\pi^2)}{9}\frac{(\max(\sqrt{\alpha},\sqrt{\beta})-n)^2}{n^2}  &&\mbox{for even } 
			n,&\\
			&\frac{8n^2(n^2+1)}{(n-1)^4}\frac{(\sqrt{\alpha}-n)^2}{n^2}  &&\mbox{for odd } 
			n \geq 3 \mbox{ with } \alpha\geq n^2, \\
			&\frac{10n^2(n^2+1)}{(n+1)^4}\frac{(\sqrt{\beta}-n)^2}{n^2} &&\mbox{for odd } 
			n \geq 3 \mbox{ with } \beta>n^2,
		\end{aligned}
		\right.
	\end{equation}
	see the estimates (3.2), (3.4), (3.5), (3.6) in \cite[Section 3]{bb1}.
	In view of \eqref{eq:estim}, if we chose $N=\emptyset$, then Theorem \ref{thm1} is an improvement of \cite[Theorem 1.4]{bb1}.
	
	Let us summarize a few properties of the function $E$ defined in Theorem \ref{thm1}, see the end of Section \ref{sec:proof} for discussion.
	\begin{lemma}\label{lem:E}
		The function $E$ has the following properties:
		\begin{enumerate}[label={\rm(\roman*)}]
			\item\label{lem:E:1} $E$ is continuous in $[4,9)$.
			\item\label{lem:E:2} Each summand in the definition \eqref{eq:E} of $E$ is strictly increasing in $[4,9)$.
			\item\label{lem:E:3} We have $E(4)=0$ and $E(6.49278\ldots)=1$.
			\item\label{lem:E:4} The infinite sum in the definition \eqref{eq:E} of $E$ in $(4,9)$ can be expressed as follows:
			\begin{align*}
				&\sqrt{\frac{6}{5}}\frac{2}{\pi}\frac{\gamma^2(\sqrt{\gamma}-2)}{\sqrt{\gamma}-1}\sum_{k=5}^\infty
				\frac{1}{(k^2-\gamma)((k-1)\sqrt{\gamma}-k)((k+1)\sqrt{\gamma}-k)}
				\\
				&=
				\sqrt{\frac{6}{5}}\frac{2}{\pi}\frac{\sqrt{\gamma}}{\sqrt{\gamma}-1}
				\sum_{k=5}^\infty\left(\frac{1}{k^2-\gamma}-\frac{1}{k^2-\frac{\gamma}{(\sqrt{\gamma}-1)^2}}\right)
				\\
				&=
				\sqrt{\frac{6}{5}}\frac{1}{\pi(\sqrt{\gamma}-1)}
				\left(
				\pi (\sqrt{\gamma}-1) \cot\left(\frac{\pi \sqrt{\gamma}}{\sqrt{\gamma}-1}\right) - \pi \cot(\pi \sqrt{\gamma})
				-
				(\sqrt{\gamma}-2)
				\right)
				\\
				&-
				\sqrt{\frac{6}{5}}\frac{2}{\pi}\frac{\gamma^2(\sqrt{\gamma}-2)}{\sqrt{\gamma}-1}\sum_{k=1}^4
				\frac{1}{(k^2-\gamma)((k-1)\sqrt{\gamma}-k)((k+1)\sqrt{\gamma}-k)}.
			\end{align*}
		\end{enumerate}
	\end{lemma}
	
	The interval $[4,9)$ appears naturally in the proof of Theorem \ref{thm1}. 
	In fact, Lemma \ref{lem:E}~\ref{lem:E:3} indicates that the highest possible value of $\sup_{n \in N} \left\{4\max(\alpha(n),\beta(n))/n^2\right\}$ to satisfy the assumption \eqref{eq:sum<1} is even smaller than $9$. 
	
	We obtain the following practical corollary of Theorem~\ref{thm1} by applying the upper bounds \eqref{eq:Bn} for the case that $N$ is the set of all even natural numbers, see Figure \ref{fig:1}.
	\begin{corollary}\label{cor:thm1:2}
		Let $G_{\alpha,\beta}$ be a Fu\v{c}\'ik system, and $\varepsilon>0$.
		Assume that
		\begin{equation*}\label{eq:sup:up2}
			\sup_{n \in \mathbb{N} \, \text{even}} \left\{\frac{4\max(\alpha(n),\beta(n))}{n^2}\right\}
			< 6.49278\ldots
		\end{equation*}
		and
		$$
		\max(\alpha(n),\beta(n))
		\leq 
		\left(n 
		+ 
		\sqrt{c_n} n^{(1-\varepsilon)/2}\right)^2
		\quad 
		\text{for all odd}~ n \geq 3,
		$$
		where
		$$
		0 \leq c_n < \frac{1-
			E^2
			\left(
			\sup\limits_{n \in \mathbb{N} \, \text{even}} \left\{\dfrac{4\max(\alpha(n),\beta(n))}{n^2}\right\}
			\right)}{45 \left( \left(1 - \frac{1}{2^{1+\varepsilon}}\right) \zeta(1+\varepsilon) - 1\right)}
		$$
		with the Riemann zeta function $\zeta$.
		Then $G_{\alpha,\beta}$ is a Riesz basis in $L^2(0,\pi)$.
	\end{corollary}

	\begin{figure}[h!]
		\centering
		\subfloat[][]
		{\includegraphics[width=0.31\linewidth]{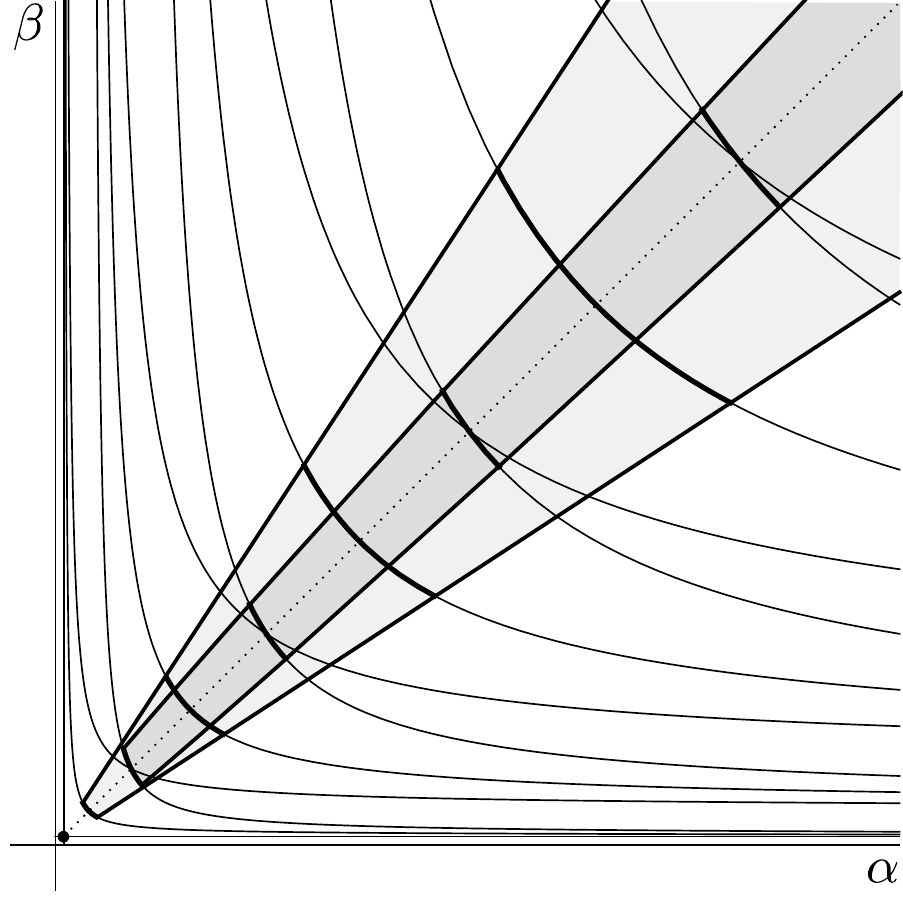}
			\label{fig:A}}
		\qquad\qquad
		\subfloat[][]
		{\includegraphics[width=0.31\linewidth]{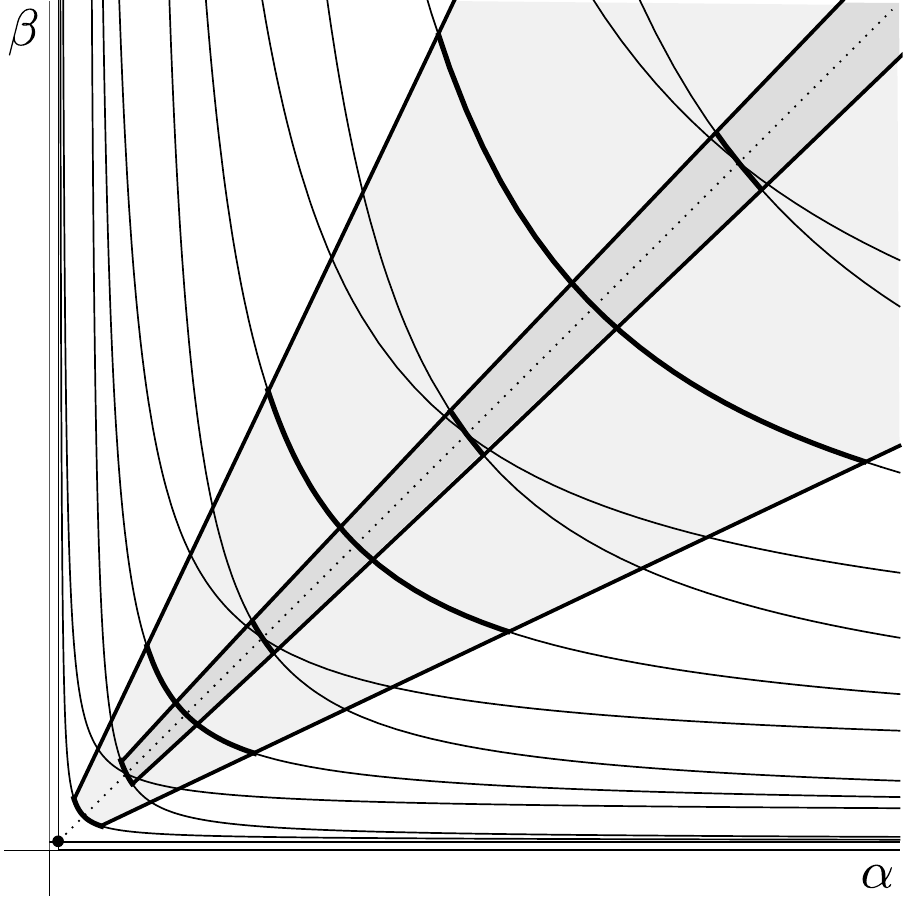}
			\label{fig:B}}
		\caption{The assumptions of Corollary \ref{cor:thm1:2} are satisfied for $(\alpha(n),\beta(n))$ belonging to bold lines inside the shaded regions.
			We have $\varepsilon=0.5$ for both panels and $\sup_{n \in \mathbb{N} \, \text{even}} \left\{\frac{4\max(\alpha(n),\beta(n))}{n^2}\right\}=5,\, 6$ in panel (A), (B), respectively.}
		\label{fig:1}
	\end{figure}
	
	If we assume that the first sum of \eqref{eq:sum<1} in Theorem \ref{thm1} is vanishing, which corresponds to $c_n=0$ for all odd $n \geq 3$ in the previous corollary, we  obtain the following result.
	\begin{corollary}\label{cor:thm1}
		Let $G_{\alpha,\beta}$ be a Fu\v{c}\'ik system such that $g^n_{\alpha,\beta}=\varphi_n$ for any odd $n$.
		Assume that
		\begin{equation}\label{eq:sup:up}
			\sup_{n \in \mathbb{N} \, \text{even}} \left\{\frac{4\max(\alpha(n),\beta(n))}{n^2}\right\}
			<
			6.49278\ldots
		\end{equation}
		Then $G_{\alpha,\beta}$ is a Riesz basis in $L^2(0,\pi)$.
	\end{corollary}
	
	\begin{figure}[h!]
		\center{\includegraphics[width=0.32\linewidth]{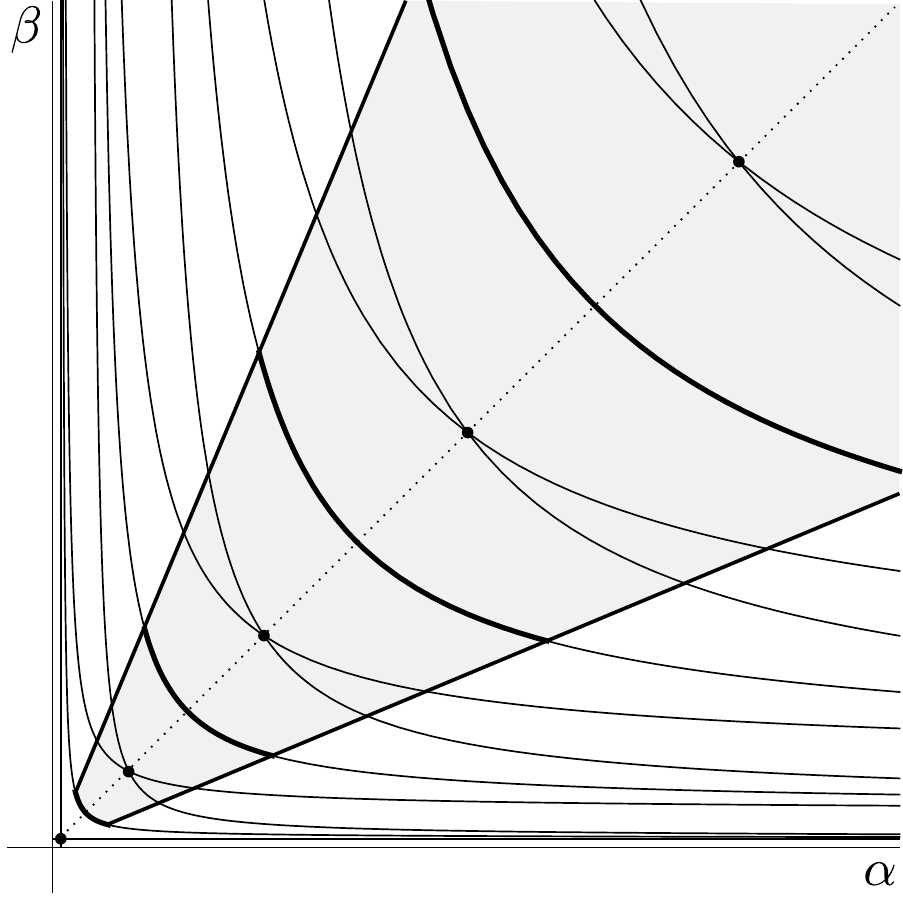}}
		\caption{The assumption \eqref{eq:sup:up} is satisfied for $(\alpha(n),\beta(n))$ belonging to bold lines inside the shaded region.}
		\label{fig1}
	\end{figure}
	
	We remark that Corollaries \ref{cor:thm1:2} and \ref{cor:thm1} are significant improvements of \cite[Theorem 1.9]{bb1} since each point $(\alpha(n), \beta(n)) \in \Gamma_n$ for even $n \geq 2$ is free to belong to the whole angular sector in between the line  
	$$
	\beta = 
	\left(\sqrt{\sup_{n \in \mathbb{N} \, \text{even}} \left\{\frac{4\max(\alpha(n),\beta(n))}{n^2}\right\}}-1\right)^{-2}\alpha
	$$
	and its reflection with respect to the main diagonal $\alpha=\beta$, and the angle of that sector is allowed to be larger than the one provided by \cite[Theorem 1.9]{bb1}. 
	We refer to Figure \ref{fig1} for the domain on the $(\alpha,\beta)$-plane given by Corollary \ref{cor:thm1}.
	Moreover, Corollary \ref{cor:thm1:2} improves \cite[Theorem 1.9]{bb1} in the sense that  $g^n_{\alpha,\beta}$ for odd $n \geq 3$ might differ from $\varphi_n$, see Figure \ref{fig:1}.

	\section{Basisness criterion}\label{sec:separ}
	
	In this section, we formulate a useful generalization of the separation of variables approach of \cite{duff} in a real Hilbert space $X$.
	The provided criterion will be applied to the space $L^2(0,\pi)$ to prove our main result, Theorem \ref{thm1}, in the subsequent section.
	
	\begin{theorem}\label{thm:main1}
		Let $M \in \mathbb{N}$.
		Let $N_\ast,N_m\subset\mathbb{N}$, $1\leq m\leq M$, be pairwise disjoint sets which form a decomposition of the natural numbers, i.e.,
		$$
		N_\ast\cup\bigcup_{m=1}^MN_m=\mathbb{N}.
		$$
		Let $\{\phi_n\}$ be a complete orthonormal sequence in $X$ and $\{f_n\}\subset X$ be a  sequence that can be represented as
		\begin{equation}\label{eq:rep}
			f_n=\phi_n+\sum_{k=1}^{\infty}C^m_{n,k}T^m_k\phi_n
			\quad\mbox{for every }n\in N_m,\,1\leq m\leq M,
		\end{equation}
		and satisfies
		\begin{equation*}
			\Lambda_\ast 
			:= 
			\left(
			\sum_{n\in N_\ast}
			\left[
			1 - \frac{\langle f_n, \phi_n\rangle^2}{\|f_n\|^2}
			\right]
			\right)^{\frac12}
			<\infty.
		\end{equation*}
		In the representation formula \eqref{eq:rep}, $\{T^m_k\}$ is a family of bounded linear mappings from $X$ to itself with bounds $\|T^m_k\|_\ast\leq t^m_k$ on the operator norm and $\{C^m_{n,k}\}$ is a family of constants with uniform bounds $|C^m_{n,k}|\leq c^m_k$ that satisfy
		\begin{equation}\label{eq:Lm}
			\Lambda_m:=\sum_{k=1}^{\infty}c^m_kt^m_k<\infty.
		\end{equation}
		Then $\{f_n\}$ is a basis in $X$ provided that
		\begin{equation}\label{eq:riesz}
			\Lambda_\ast^2 + \sum_{m=1}^M\Lambda_m^2 < 1.
		\end{equation}
		If, in addition, the subsequence $\{f_n\}_{n \in N_*}$ is bounded, then $\{f_n\}$ is a Riesz basis in $X$.
	\end{theorem}
	
	\begin{proof}
		Denote $\widetilde{f}_n = \rho_n f_n$, where $\rho_n=1$ for $n \in \mathbb{N} \setminus N_*$, and the values of $\rho_n$ for $n \in N_*$ will be specified later.
		Let $\{a_n\}_{n \in \widetilde{N}}$ be an arbitrary finite sequence of constants with a finite index set $\widetilde{N}\subset\mathbb{N}$. 
		Setting $\widetilde{N}_\ast=N_\ast\cap\widetilde{N}$ and $\widetilde{N}_m=N_m\cap\tilde{N}$ for every $1\leq m\leq M$, we obtain
		\begin{equation}\label{eq:norm-estim}
			\big\|\sum_{n\in \widetilde{N}} a_n(\widetilde{f}_n-\phi_n)\big\|
			\leq
			\sum_{m=1}^M\big\|\sum_{n\in\widetilde{N}_m}a_n(f_n-\phi_n)\big\|
			+
			\big\|\sum_{n\in\widetilde{N}_\ast}a_n(\rho_n f_n-\phi_n)\big\|.
		\end{equation}
		For the first sum on the right-hand side of \eqref{eq:norm-estim}, we apply the representation \eqref{eq:rep} and get
		\begin{align*}
			\sum_{m=1}^M
			\big\|\sum_{n\in\widetilde{N}_m}a_n(f_n-\phi_n)\big\|
			&=
			\sum_{m=1}^M\big\|\sum_{n\in\widetilde{N}_m}a_n\sum_{k=1}^{\infty}C^m_{n,k}T^m_k\phi_n\big\|
			=
			\sum_{m=1}^M\big\|\sum_{k=1}^{\infty}T^m_k\sum_{n\in\widetilde{N}_m}C^m_{n,k}a_n\phi_n\big\|
			\\
			&\leq
			\sum_{m=1}^M\sum_{k=1}^{\infty}\big\|T^m_k\sum_{n\in\widetilde{N}_m}C^m_{n,k}a_n\phi_n\big\|
			\leq
			\sum_{m=1}^M\sum_{k=1}^{\infty}t^m_k\big\|\sum_{n\in\widetilde{N}_m}C^m_{n,k}a_n\phi_n\big\|
			\\
			&\leq
			\sum_{m=1}^M\sum_{k=1}^{\infty}t^m_kc^m_k\big\|\sum_{n\in\widetilde{N}_m}a_n\phi_n\big\|
			=
			\sum_{m=1}^M\Lambda_m\big\|\sum_{n\in\widetilde{N}_m}a_n\phi_n\big\|,
		\end{align*}
		while for the second sum we obtain
		\begin{align*}
			\big\|\sum_{n\in\widetilde{N}_\ast}a_n(\rho_n f_n-\phi_n)\big\|
			\leq
			\Big(\sum_{n\in\widetilde{N}_\ast}\|\rho_nf_n-\phi_n\|^2\Big)^{\frac12}
			\Big(\sum_{n\in\widetilde{N}_\ast}|a_n|^2\Big)^{\frac12}.
		\end{align*}
		Let us choose $\rho_n$ to be a minimizer of the distance
		$\|\rho f_n-\phi_n\|^2$ with respect to $\rho$.
		Since
		$$
		\|\rho f_n-\phi_n\|^2 = \rho^2 \|f_n\|^2 - 2 \rho \langle f_n, \phi_n\rangle + 1,
		$$
		we readily see that 
		$$
		\|\rho_n f_n-\phi_n\|^2
		=
		\min_{\rho \in \mathbb{R}} 
		\|\rho f_n-\phi_n\|^2
		=
		1 - \frac{\langle f_n, \phi_n\rangle^2}{\|f_n\|^2}
		=
		\|f_n-\phi_n\|^2
		-
		\frac{(\|f_n\|^2-\langle f_n,\phi_n \rangle)^2}{\|f_n\|^2}
		$$
		with $\rho_n = \langle f_n, \phi_n\rangle/\|f_n\|^2$.
		Evidently, we have $|\rho_n| \leq 1$.
		We remark that in case of $\rho_n=0$, we get $\Lambda_\ast\geq1$ which violates the assumption \eqref{eq:riesz}. 
		Applying now the Cauchy inequality, we deduce from \eqref{eq:norm-estim} that
		\begin{align*}
			\big\|\sum_{n\in\widetilde{N}} a_n(\widetilde{f}_n-\phi_n)\big\|
			&\leq
			\sum_{m=1}^M\Lambda_m\big\|\sum_{n\in\widetilde{N}_m}a_n\phi_n\big\|
			+
			\Lambda_\ast
			\Big(\sum_{n\in\widetilde{N}_\ast}|a_n|^2\Big)^{\frac12}
			\\
			&\leq 
			\Big(\sum_{m=1}^M\Lambda_m^2+\Lambda_\ast^2\Big)^{\frac12}
			\big\|\sum_{n\in \widetilde{N}}a_n\phi_n\big\|.
		\end{align*}	
		We conclude from the assumption \eqref{eq:riesz} that the sequence $\{\widetilde{f}_n\}$ is Paley-Wiener near to the complete orthonormal sequence $\{\phi_n\}$ and, thus, it is a Riesz basis in $X$, see, e.g., \cite[Chapter~1, Theorem 10]{young}.
		Clearly, $\{f_n\}=\{\rho_n^{-1}\widetilde{f}_n\}$ is a basis in $X$. 
		Assume that the subsequence $\{f_n\}_{n \in N_*}$ is bounded. 
		Then there exists $0<c<1$ such that $|\rho_n| \geq c$ for all $n \in \widetilde{N}_*$. 
		This is evident for finite $N_*$ since $\rho_n \neq 0$.
		In the case of infinte $N_*$, if we suppose that $\rho_n$ goes to zero up to a subsequence, then the sum
		$$
		\Lambda_\ast 
		= 
		\left(
		\sum_{n\in N_\ast}
		\left[
		1 - \frac{\langle f_n, \phi_n\rangle^2}{\|f_n\|^2}
		\right]
		\right)^{\frac12}
		=
		\left(
		\sum_{n\in N_\ast}
		\left[
		1 - \rho_n^2\|f_n\|^2
		\right]
		\right)^{\frac12}
		$$
		does not converge. Recalling $\rho_n=1$ for every $n\in\mathbb{N}\setminus N_{\ast}$, we obtain $1\leq |\rho_n^{-1}|\leq c^{-1}$ for all $n \in \mathbb{N}$ which implies that $\{f_n\}$ is a Riesz basis in $X$, see, e.g., \cite[Chapter 1, Theorem 9]{young}.
	\end{proof}
	
	In the case $N_1=\mathbb{N}$, Theorem \ref{thm:main1} simplifies to Theorem D from \cite{duff} and for $N_\ast=\mathbb{N}$ we get the result of Theorem V-2.21 and Corollary V-2.22 i) from \cite{kato} which were discussed in \cite{bb1}.
	
	\begin{remark}
		It can be seen from the proof of Theorem \ref{thm:main1} that if we weaken the definition of $\Lambda_\ast$ to
		$$
		\widetilde{\Lambda}_\ast
		:=
		\left(\sum_{n\in N_\ast}\|f_n-\phi_n\|^2\right)^{\frac12}
		\leq
		\Lambda_\ast,
		$$
		then we can formulate the following result under the assumptions of Theorem \ref{thm:main1}: 
		the sequence $\{f_n\}$ is a Riesz basis in $X$ provided that
		$$
		\widetilde{\Lambda}_\ast^2+\sum_{m=1}^{M}\Lambda_m^2<1.
		$$
		The boundedness of the subsequence $\{f_n\}_{n\in N_\ast}$ is not required under this modified assumption.
	\end{remark}

	\section{Proof of Theorem \ref{thm1}}\label{sec:proof}
	
	We prove Theorem \ref{thm1} by applying the general basisness criterion introduced in the previous section. To determine the bounds on the family of constants $\{C^m_{n,k}\}$ in Theorem \ref{thm:main1} we will make use of the Fourier coefficients of Fu\v{c}\'ik eigenfunctions corresponding to Fu\v{c}\'ik eigenvalues on the first nontrivial curve $\Gamma_2$. 
	Namely, we provide estimates for the Fourier coefficients of the odd Fourier expansion of the function
	$$
	g_{\gamma,\gamma/(\sqrt{\gamma}-1)^2}^2=\sum_{k=1}^{\infty}A_k(\gamma)\varphi_k(x)
	$$
	for $\gamma>4$ which are given by
	\begin{align*}
		A_k(\gamma) = 
		\int_0^{\pi}g_{\gamma,\gamma/(\sqrt{\gamma}-1)^2}^2(x)\varphi_k(x)\,\mathrm{d}x  
		= \frac2{\pi}\frac{\gamma^2}{\sqrt{\gamma}-1}\frac{(2-\sqrt{\gamma})\sin\left(\frac{k\pi}{\sqrt{\gamma}}\right) }{(k^2-\gamma)(k^2(\sqrt{\gamma}-1)^2-\gamma)},
	\end{align*}
	and of the function
	$$
	g_{\delta/(\sqrt{\delta}-1)^2,\delta}^2=\sum_{k=1}^{\infty}\widetilde{A}_k(\delta)\varphi_k(x)
	$$
	for $\delta>4$ which are given by
	\begin{align*}
		\widetilde{A}_k(\delta) = 
		\int_0^{\pi}g_{\delta/(\sqrt{\delta}-1)^2,\delta}^2(x)\varphi_k(x)\,\mathrm{d}x  
		=
		(-1)^k A_k(\delta).
	\end{align*}
	In the case $\gamma=\delta=4$, we have $A_2=1$ and $A_k=0$ for any other $k \in \mathbb{N}$.

	Obviously, we have
	\begin{equation}\label{eq:A1}
		|A_1(\gamma)| = 
		B_1(\gamma):=
		\frac2{\pi}\frac{\gamma^2}{\sqrt{\gamma}-1}\frac{(\sqrt{\gamma}-2)\sin\left(\frac{\pi}{\sqrt{\gamma}}\right) }{(\gamma-1)(2\sqrt{\gamma}-1)}
	\end{equation}
	and it was shown in \cite[Section 5]{bb1} that
	\begin{equation}\label{eq:A2}
		|A_2(\gamma)-1|
		\leq
		B_2(\gamma):=
		\frac{((3+\pi^2)\gamma+(9-2\pi^2)\sqrt{\gamma}-6)(\sqrt{\gamma}-2)}
		{3(\sqrt{\gamma}-1)(\sqrt{\gamma}+2)(3\sqrt{\gamma}-2)}.
	\end{equation}
	For $\gamma \in [4,9)$, we clearly have
	\begin{equation}\label{eq:A3}
		|A_3(\gamma)|
		=
		B_3(\gamma):=
		\frac2{\pi}\frac{\gamma^2}{\sqrt{\gamma}-1}\frac{(\sqrt{\gamma}-2)\left(-\sin\left(\frac{3\pi}{\sqrt{\gamma}}\right)\right)}{(9-\gamma)(2\sqrt{\gamma}-3)(4\sqrt{\gamma}-3)}
	\end{equation}
	and for $k\geq 4$ we use the simple estimate
	\begin{equation}\label{eq:Ak}
		|A_k(\gamma)| 
		\leq
		B_k(\gamma):=
		\frac2{\pi}\frac{\gamma^2}{\sqrt{\gamma}-1}\frac{(\sqrt{\gamma}-2)}{(k^2-\gamma)((k-1)\sqrt{\gamma}-k)((k+1)\sqrt{\gamma}-k)}.
	\end{equation}
	Evidently, the same bounds hold for $\widetilde{A}_k$.
	Numerical calculations with the exact coefficients show that the used estimates in \eqref{eq:A2} and \eqref{eq:Ak} do not influence the results in a significant way. 
	
	\begin{lemma}\label{lem:increas}
		Let $\gamma \in [4,9)$ and $k \in \mathbb{N}$.
		Then $B_k$ is strictly increasing.
	\end{lemma}
	\begin{proof}
		For simplicity, we introduce the change of variables $x=\sqrt{\gamma} \in [2,3)$. 	
		The first derivative of $B_k(x^2)$ with $k\in\mathbb{N}\setminus\{1,3\}$ is a rational function with a positive denominator and we can easily check that the numerator is positive, as well. 
		Hence, $B_k(\gamma)$ with $k\in\mathbb{N}\setminus\{1,3\}$ is strictly increasing for $\gamma \in [4,9)$.
		The first derivative of $B_1(x^2)$ takes the form
		$$
		\frac{2x^2(x-1)\cos\left(\frac{\pi}{x}\right)\left[
			x(2x^4-4x^3-x^2+15x-8)\tan\left(\frac{\pi}{x}\right)
			-
			\pi(2x^4-5x^3+5x-2)
			\right]}{\pi(x-1)^2(x^2-1)^2(2x-1)^2}.
		$$
		Noting that $x(2x^4-4x^3-x^2+15x-8)>0$ for $x \in [2,3)$, we can use the simple lower bound $\tan\left(\frac{\pi}{x}\right) \geq \sqrt{3}$ to show that the expression in square brackets is positive. 
		Since all other terms in the derivative are also positive, we conclude that $B_1(\gamma)$ is strictly increasing for $\gamma \in [4,9)$.
		
		Finally, the numerator of the first derivative of $B_3(x^2)$ is given by
		\begin{align}
			\notag
			-2x^2
			\bigg[
			&x(10x^5+90x^4-765x^3+1872x^2-1863x+648)\sin\left(\frac{3\pi}{x}\right)
			\\
			\label{eq:A3-est}
			&+3\pi(8x^6-42x^5+7x^4+315x^3-693x^2+567x-162)\cos\left(\frac{3\pi}{x}\right)
			\bigg],
		\end{align}
		whereas the denominator is a positive polynomial. 
		We have $\sin\left(\frac{3\pi}{x}\right) < 0$ and $\cos\left(\frac{3\pi}{x}\right) < 0$ for $x \in [2,3)$, and taking into account that
		\begin{align*}
			x(10x^5+90x^4-765x^3+1872x^2-1863x+648)&<0,\\
			3\pi(8x^6-42x^5+7x^4+315x^3-693x^2+567x-162)&>0,
		\end{align*}
		we employ the estimates
		$$
		\sin\left(\frac{3\pi}{x}\right) 
		<
		-\left( \frac{3\pi}{x}-\pi\right)+\frac{1}{6}\left(\frac{3\pi}{x}-\pi\right)^3
		\quad 
		\text{and}
		\quad
		\cos\left(\frac{3\pi}{x}\right) > -1.
		$$
		As a result, the expression \eqref{eq:A3-est} is estimated from below by a polynomial which is positive for $x \in [2,3)$.
		Thus, $B_3(\gamma)$ is strictly increasing for $\gamma \in [4,9)$.
	\end{proof}
	
	Now we are ready to prove our main result.
	\begin{proof}[Proof of Theorem \ref{thm1}]
		We apply Theorem \ref{thm:main1}, where we consider $X=L^2(0,\pi)$, the sequence $\{f_n\}$ is the Fu\v{c}\'{i}k system, which is bounded by definition, and the complete orthonormal set $\{\phi_n\}$ is given by $\{\varphi_n\}$. We set $M=1$ and $N_1=N$ and choose $N_\ast = \mathbb{N} \setminus N$ as assumed in Theorem~\ref{thm1}. 
		We define the linear operators $T^1_k\colon L^2(0,\pi)\to L^2(0,\pi)$ as 
		\begin{equation*}\label{eq:TK}
			T^1_k g(x) = g^*\left(\frac{kx}{2}\right),
		\end{equation*}
		where 
		$$
		g^*(x) = (-1)^\kappa g(x-\pi \kappa) \quad \text{for}~~ \pi \kappa \leq x \leq \pi (\kappa+1), \quad \kappa \in \mathbb{N} \cup \{0\},
		$$
		is the $2\pi$-antiperiodic extension for arbitrary functions $g \in L^2(0,\pi)$.
		In particular, we have $T^1_k \sin(nx) = \sin\left(\frac{knx}{2}\right)$ for every even $n$.	
		It was proven in \cite[Appendix B]{bb1} that  $\|T^1_k\|_* = 1$ for even $k$ and $\|T^1_k\|_* = \sqrt{1+1/k}$ for odd $k$.
		
		Let $n\in N$ be fixed and recall that $n$ is even. 
		To begin with, we assume that $\alpha(n)> n^2$. 
		The Fu\v{c}\'{i}k eigenfunction $g^n_{\alpha,\beta}$ has the dilated structure
		$$
		g^n_{\alpha,\beta}(x)
		=
		g_{\gamma_n,\gamma_n/(\sqrt{\gamma_n}-1)^2}^2\left(\frac{nx}{2}\right)
		\quad\mbox{with}\quad
		\gamma_n = \frac{4\alpha(n)}{n^2}
		$$
		and, thus, has the odd Fourier expansion
		$$
		g^n_{\alpha,\beta}(x)
		=
		g_{\gamma_n,\gamma_n/(\sqrt{\gamma_n}-1)^2}^2\left(\frac{nx}{2}\right)
		=
		\sum_{k=1}^{\infty}A_k(\gamma_n)\varphi_k\left(\frac{nx}{2}\right)
		=
		\sum_{k=1}^{\infty}A_k(\gamma_n)T^1_k\varphi_n(x).
		$$
		From this, we directly see that the representation \eqref{eq:rep} of $g^n_{\alpha,\beta}$ in terms of $\{\varphi_n\}$ holds with the constants $C^1_{n,k}=A_k(\gamma_n)$ for $k\neq 2$ and $C^1_{n,2}=1-A_2(\gamma_n)$. 
		The bounds for the constants $|C^1_{n,k}|$ are given by the functions $B_k(\gamma_n)$ defined in \eqref{eq:A1}, \eqref{eq:A2}, \eqref{eq:A3}, and \eqref{eq:Ak}, which are strictly increasing in the interval $[4,9)$ by Lemma \ref{lem:increas}. 
		For the case $\beta(n)> n^2$, the Fu\v{c}\'ik eigenfunction has the form
		$$
		g^n_{\alpha,\beta}(x)
		=
		g_{\delta_n/(\sqrt{\delta_n}-1)^2,\delta_n}^2\left(\frac{nx}{2}\right)
		\quad\mbox{with}\quad
		\delta_n = \frac{4\beta(n)}{n^2},
		$$
		and by analogous arguments 	we get the bounds $|C^1_{n,k}| \leq B_k(\delta_n)$.
		If $\alpha(n)=n^2$, and hence $\beta(n)=n^2$, then we set $C^1_{n,k}=0$ for every $k \in \mathbb{N}$.
		
		In view of the monotonicity, we have
		$$
		|C^1_{n,k}|
		\leq
		B_k\left(\sup_{n\in N} \max(\gamma_n,\delta_n)\right)
		.
		$$
		Therefore, we can provide the following upper estimate on the constant $\Lambda_1$ defined in \eqref{eq:Lm}:
		\begin{align*}
			\Lambda_1
			&\leq
			\sqrt{2} B_1\left(\sup_{n\in N} \max(\gamma_n,\delta_n)\right)
			+
			B_2\left(\sup_{n\in N} \max(\gamma_n,\delta_n)\right)
			\\
			&+
			\sqrt{\frac{4}{3}}B_3 \left(\sup_{n\in N} \max(\gamma_n,\delta_n)\right)
			+
			B_4\left(\sup_{n\in N} \max(\gamma_n,\delta_n)\right)	
			\\
			&+
			\sqrt{\frac{6}{5}}
			\sum_{k=5}^\infty
			B_k\left(\sup_{n\in N} \max(\gamma_n,\delta_n)\right)
			\\
			&=
			E\left(\sup_{n\in N} \max(\gamma_n,\delta_n)\right)
			=
			E\left(
			\sup_{n \in N} \left\{\frac{4\max(\alpha(n),\beta(n))}{n^2}\right\}
			\right),
		\end{align*}
		with the function $E$ introduced in Theorem \ref{thm1}, and $E$ is strictly increasing in $[4,9)$. 
		Noticing that we have
		$$
		\Lambda_\ast
		=
		\left(\sum_{n\in N_\ast}
		\left[
		1 - \frac{\langle g^n_{\alpha,\beta}, \varphi_n\rangle^2}{\|g^n_{\alpha,\beta}\|^2}
		\right]\right)^{\frac12},
		$$
		the assumption \eqref{eq:sum<1} yields the assumption $\Lambda_\ast^2+\Lambda_1^2<1$ in Theorem \ref{thm:main1}. 
		This completes the proof of Theorem \ref{thm1}.
	\end{proof}

	We conclude this note by discussing Lemma \ref{lem:E}.
	The monotonicity statement \ref{lem:E:2} directly follows from Lemma \ref{lem:increas}, and to obtain the alternative representation \ref{lem:E:4}, we make use of the identity
	$$
	\sum_{k=1}^\infty
	\frac{1}{k^2-a^2}
	=
	\frac{1}{2a^2}
	-
	\frac{\pi \cot(\pi a)}{2a},
	\quad a \not\in \mathbb{N},
	$$
	see, e.g., \cite[(6.3.13)]{as}.
	The representation \ref{lem:E:4} shows that the function $E$ is continuous in $[4,9)$.
	The combination of the continuity and monotonicity of $E$ allows us to compute values of $E$ with an arbitrary precision.
	In particular, we have $E(6.49278\ldots)=1$.

\end{document}